\documentclass[reqno]{amsart}
\usepackage[utf8]{inputenc}

\usepackage{todonotes}

\usepackage{amsfonts,color,newclude}
\usepackage{amsthm}
\usepackage{amssymb,transparent}
\usepackage{amsmath}
\usepackage{amscd}
\usepackage{svg}

\usepackage{graphicx}

\usepackage{tikz}
\usepackage[style=alphabetic, backend=biber, maxbibnames=50]{biblatex}
\newtheorem{theorem}{Theorem}[section]

\newtheorem{proposition}[theorem]{Proposition}
\newtheorem{lemma}[theorem]{Lemma}
\newtheorem{cor}[theorem]{Corollary}

\newtheorem{conjecture}[theorem]{Conjecture}

\theoremstyle{definition}
\newtheorem{defn}[theorem]{Definition}
\newtheorem{claim}[theorem]{Claim}
\newtheorem{question}[theorem]{Question}
\newtheorem{construction}[theorem]{Construction}

\theoremstyle{remark}

\newtheorem{example}[theorem]{Example}

\begin{document}

\title{Graphs with many hamiltonian paths}
\author[E. Carlson]{Erik Carlson}
\author[W. Fletcher]{Willem Fletcher}
\author[M.K. Montee]{MurphyKate Montee}
\author[C. Nguyen]{Chi Nguyen}
\author[J. Renders]{Jarne Renders}
\author[X. Zhang]{Xingyi Zhang}

\maketitle

\begin{abstract}

A graph is \emph{hamiltonian-connected} if every pair of vertices can be connected by a hamiltonian path, and it is \emph{hamiltonian} if it contains a hamiltonian cycle. We construct families of non-hamiltonian graphs for which the ratio of pairs of vertices connected by hamiltonian paths to all pairs of vertices approaches 1. We then consider minimal graphs that are hamiltonian-connected. It is known that any order-$n$ graph that is  hamiltonian-connected must have $\geq 3n/2$ edges. We construct an infinite family of graphs realizing this minimum. 

\end{abstract}

\section{Introduction}
A hamiltonian path in a graph is a path which visits every vertex exactly once. A graph is \emph{hamiltonian} if it admits a hamiltonian cycle, and \emph{homogeneously traceable} if every vertex of $G$ is the starting vertex of a hamiltonian path. If every pair of vertices in $G$ is connected with a hamiltonian path, then $G$ is \emph{hamiltonian-connected}. This class of graphs was introduced by Ore \cite{Ore} in 1963, and has been well studied since then. Every hamiltonian-connected graph must be hamiltonian (as long as it has at least 3 vertices), and every hamiltonian graph must be homogeneously traceable. 

It is natural to explore the reverse implications; more generally, one might ask `how many hamiltonian paths can a non-hamiltonian graph contain?' Every cycle graph is hamiltonian but not hamiltonian-connected, and the Petersen graph is (perhaps unsurprisingly) an example of a homogeneously traceable non-hamiltonian (HTNH) graph. In fact, Chartrand, Gould, and Kapoor show in \cite{Chartrand} that there exist HTNH graphs of any order $n\geq 9$. Hu and Zhan extended this in \cite{HU21} to find families of regular HTNH graphs. In all of these papers, the authors construct graphs by starting with some base graph and building new graphs (with larger vertex sets) inductively. 

In this paper, we consider graphs with `many' hamiltonian paths, in the sense that the number of pairs of vertices connected by a hamiltonian path is large. We say that a graph $G$ is \emph{$k$-pair-strung} if $k$ pairs of vertices can be connected by a hamiltonian path. Thus a hamiltonian-connected graph of order $n$ is $\binom{n}{2}$-pair-strung, and a homogeneously traceable graph is (at least) $(n/2)$-pair-strung.

In Section \ref{sec: highly pair strung graphs} we consider non-hamiltonian graphs which are nevertheless highly pair-strung; in particular, we let $r_G$ denote the fraction of pairs of vertices of $G$ which are connected by a hamiltonian path, so that in a graph $G$ with order $n$, 
    \[
        r_G = \max\left.\left\{\frac{k}{\binom{n}{2}} \;\right|\; G\mbox{ is } k\mbox{-pair-strung}\right\}.
    \]
We first show that there exists a family of non-hamiltonian graphs $\mathcal{G}$ for which $\sup_{G \in \mathcal{G}}\{r_G\} = 1$. We then investigate possible values of $r_G$; we describe an explicit construction that produces families of graphs for which $r_G$ approaches $\frac{k-1}{k}$. This construction depends on finding a collection of edges in a non-hamiltonian graph for which every edge is connected by a hamiltonian path containing the other edges. We call such edge sets \textit{H-path connected} (see Definition \ref{def:Hpathconnected}, Theorem \ref{thm: highly pair strung family}). 

In Section \ref{sec: minimal hc graphs} we consider another extreme; rather than limiting ourselves to non-hamiltonian graphs, we construct an infinite family of hamiltonian graphs containing a minimal edge set (see Theorem \ref{thm: minimal hamiltonian-connected}).


\subsection{Acknowledgements}
The authors would like to thank Xingzhi Zhan for alerting them to \cite{Chartrand, Skupien}; a previous version of this manuscript reproduced several of their results.  

\section{Nearly Hamiltonian-Connected Non-Hamiltonian Graphs}\label{sec: highly pair strung graphs}
In this section, we consider graphs that have many pairs of vertices connected by a hamiltonian path. We will quantify this in the following definition.

\begin{defn} \label{defn: n pair strung}
A graph $G$ is $k$-\textit{pair-strung} if $k$ pairs of the vertices in $G$ have a hamiltonian path between them.
\end{defn}

Note that the existence of HTNH graphs implies that there exist non-hamiltonian graphs which are at least $(n/2)$-pair-strung. Unsurprisingly, we can also find non-hamiltonian graphs with fewer hamiltonian paths.

\begin{example}
Let $n>m>0$. Consider the graph $G$ obtained by attaching a path with $n-m+1$ vertices, $P_{n-m+1},$ to a complete graph on $m$ vertices so that $K_m \cap P_{n-m+1} = \{v\}$ is one of the endpoints of $P_{n-m+1}$. (This is illustrated in Figure \ref{fig: spectrum}). Let $w$ be the other endpoint of $P_{n-m+1}$. Then every hamiltonian path in $G$ begins (or ends) at $w$, and ends (or begins) at one of the vertices in $V(K_m) \backslash \{v\}.$ In particular, exactly $m$ vertices of $G$ are an endpoint of a hamiltonian path, and exactly $m$ pairs of vertices of $G$ are are connected by a hamiltonian path. Furthermore, $G$ is non-hamiltonian since no cycle can contain $w$.
\end{example}

\begin{figure}
	\begin{tikzpicture}
    \draw (0,2) ellipse (1cm and 2cm);
    \draw[fill=black] (0.25,1) circle (3pt);
    \draw[fill=black] (2.25,1) circle (3pt);
    \draw[fill=black] (3.25,1) circle (3pt);
    \draw[fill=black] (5.25,1) circle (3pt);
    \path (3.6,1) -- node[auto=false]{\ldots} (4.9,1);
    \node at (0.25,1.5) {$v$};
    \node at (3.75,.5) {$P_{n-m+1}$};
    \node at (5.25,1.5) {$w$};
    \node at (-1, 4) {$K_m$};
    \draw[thick] (0.25,1) -- (3.55,1);
    \draw[thick] (4.95,1) -- (5.25,1);
	\end{tikzpicture}
	\caption{The graph illustrated here has $n$ vertices. Every vertex in $K_m$ except $v$ is an endpoint of a hamiltonian path, and all hamiltonian paths begin (or end) at $w$. }
	\label{fig: spectrum} 
\end{figure}

For a graph, $G$, we define the pair connected ratio, $r_G$, as the fraction of pairs of vertices that have a hamiltonian path between them.  More precisely, we have the following definition:

\begin{defn} \label{defn: R}
Let $G$ be a graph of order $n$. Let the \emph{pair connected ratio} of $G$ be defined as
    \[
    r_G = \max\left.\left\{\frac{k}{\binom{n}{2}} \;\right|\; G\mbox{ is } k\mbox{-pair-strung}\right\}.
    \]

For any $n$, let $R_n = \sup \{r_G | G \mbox{ is a non-hamiltonian graph of order }n\}$, and let $R = \sup \{R_n | n\geq 3\}$. 
\end{defn}

We will first establish an upper bound on $R_n$. 

\begin{theorem} \label{thm: upper-bound}
For $n > 3$, $R_n \leq \frac{n-2}{n}$.
\end{theorem}

\begin{proof}
Consider a non-hamiltonian graph, $G$, with at least one hamiltonian path $\gamma$ and $n>3$ vertices. Label the vertices $v_1, v_2, \dots, v_n$ in order along $\gamma$. If $\gamma$ is the only hamiltonian path in $G$ then 
$r_G = \frac{2}{n(n-1)}<\frac{n-2}{n}$ when $n>3$.
Suppose instead that $G$ has at least one other hamiltonian path, $\gamma' = v_{i_1}, v_{i_2}, \dots, v_{i_n}$. If $i_1 = k$ and $i_n = k\pm 1$, then $\gamma$ is a hamiltonian cycle. Thus any hamiltonian path in $G$ can not start and end on vertices with labels differing by one. The total number of pairs of vertices which are non-adjacent on $\gamma$ is at most $\frac{(n-1)(n-2)}{2}$. Dividing by the number of pairs of vertices, $\binom{n}{2}$, we get $r_G \leq \frac{n-2}{n}$. Thus $R_n \leq \frac{n-2}{n}.$
\end{proof}

It is clear from the definition that $R\leq 1$; in fact, we can show that $R = 1$ by considering \emph{maximally non-hamiltonian} (MNH) graphs; that is, a graph $G$ where $G$ is non-hamiltonian but $G + e$ is hamiltonian for every edge $e$ in the complement of $G$. 

\begin{theorem}\label{thm: R=1}
    R = 1.
\end{theorem}
    \begin{proof}
        Note that if $G$ is MNH, then every pair of non-adjacent vertices in $G$ must be connected by a hamiltonian path. Therefore an MNH graph $G$ with $n$ vertices and $m$ edges must be $\left(\binom{n}{2} - m\right)$-pair-strung. There exist MNH graphs of order $n\geq 19$ with $\lceil 3n/2\rceil$ edges (see \cite{ClarkEntShapiro, ClarkEnt, Linetal}). Let $\{G_n\}$ be a family of such graphs. Then we have:
         \[
            \lim_{n\to \infty} r_{G_n} = \lim_{n\to \infty} \frac{n(n-1)-3n}{n(n-1)} = 1.
        \]
    \end{proof}

The rest of this section offers an alternative proof of Theorem \ref{thm: R=1}, and is motivated by the following question:

\begin{question}
What are the possible values of $r_G$?
\end{question}

We conjecture that for every rational number $q \in[0,1)$ there is some graph $G$ so that $r_G = q.$ 

\begin{conjecture}
For every $n$ there exists a number $p(n)$ so that for every $k\leq p(n)$ there exists a $k$-pair-strung graph of order $n$ which is not $(k+1)$-pair-strung, and there is no $k$-pair-strung non-hamiltonian graph of order $n$ for any $k>p(n)$. 
\end{conjecture}

Using this language, we can ask further questions.

\begin{question}
How does $p(n)$ grow as a function of $n$? 
\end{question}

For small values of $n$ we have the data shown in Figure \ref{fig: conjecture data}, which were obtained by computer search.

\begin{figure}
    \centering
        \begin{tabular}{c|cccccccc}
            $n$ & 4&5&6&7&8&9&10&11\\
            \hline
            $p(n)$ &2&4&6&9&13&18&25&34
        \end{tabular}
    \caption{Value of $p(n)$ for small numbers $n$.}
    \label{fig: conjecture data}
\end{figure}

For the rest of this section we describe a construction that produces a family of graphs with $r_G \to \frac{n-1}{n}$ for any natural number $n$.

\begin{defn} \label{def:Hpathconnected} 
A set of edges $S = \{e_1, e_2, \dots, e_n\}$ in a graph $G$ is \emph{H-path connected} if for every $i \neq j$, there is a hamiltonian path with its endpoints in $e_i$ and $e_j$ and which contains every edge in $S - \{e_i, e_j\}.$
\end{defn}

Finding sets of H-path connected edges is, in general, computationally difficult. One example of such a set is any collection of five pairwise disjoint edges in the Petersen graph. A larger example is demonstrated below.

\begin{example}
Consider the graph $G$ illustrated in Figure \ref{fig: hpathex}. This graph is the smallest known \emph{almost hypohamiltonian} graph (see \cite{Zamfi}). That is, it is non-hamiltonian and for all but one vertex $v$, the graph $G-v$ is hamiltonian. (In this case, the exceptional vertex is the central vertex.) The set of seven red edges indicated in Figure \ref{fig: hpathex} is an H-path connected set of edges. The hamiltonian paths realizing this are indicated in green.
\end{example}

\begin{figure}
    \centering
    \includegraphics[width=.8\textwidth]{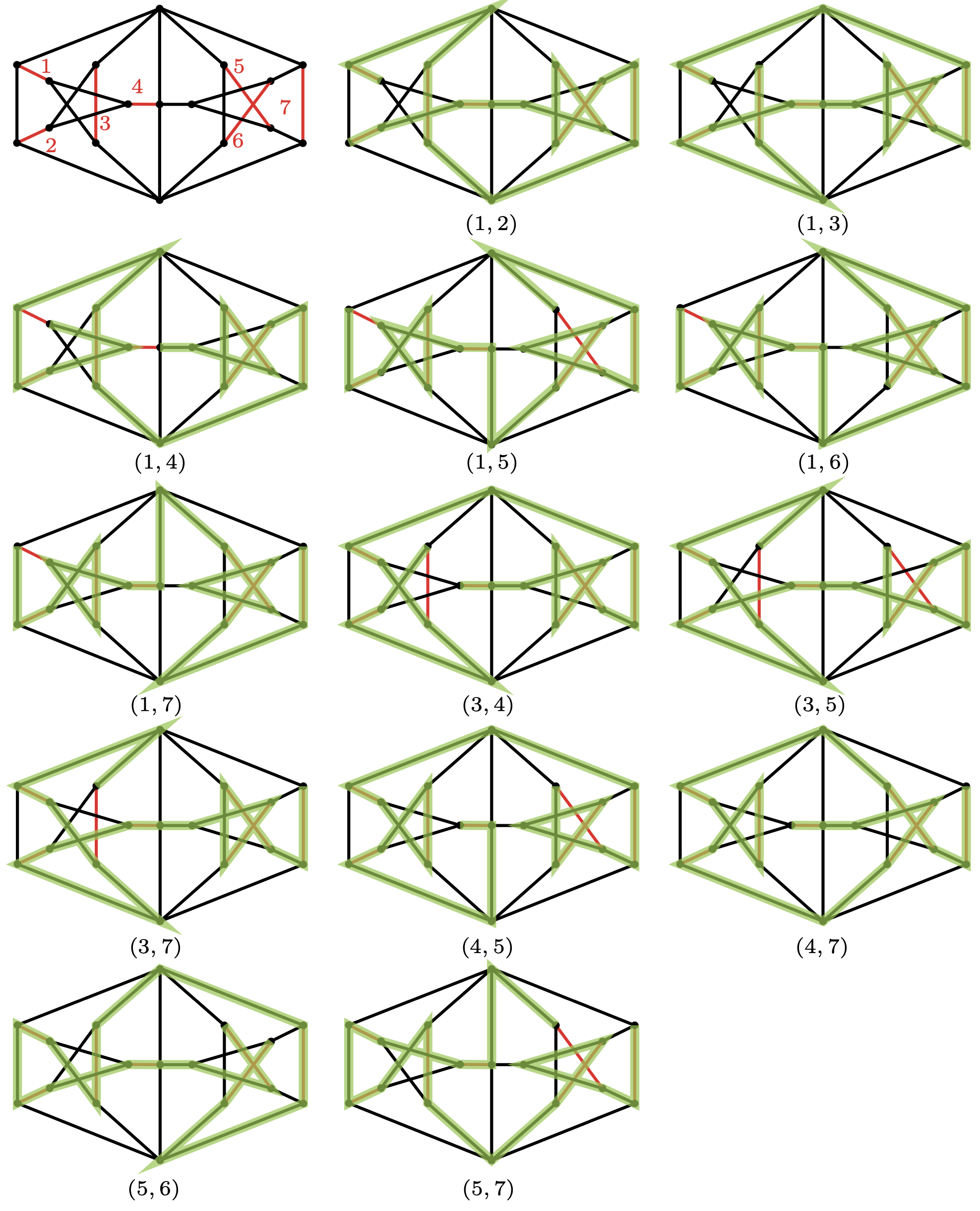}
    \caption{A graph with an H-path connected edge set of size seven (indicated in red). The hamiltonian paths connecting each pair of edges are shown in green. Note that some of the pairs of edges are not illustrated because of the reflective symmetry of the graph.}
    \label{fig: hpathex}
\end{figure}

We use a non-hamiltonian graph containing an H-path connected set of edges to construct a family of non-hamiltonian graphs which are highly pair-strung as follows. 

\begin{construction}\label{construction of Gk}
Let $G$ be a non-hamiltonian graph and let $k \in \mathbb{N}$. Let $S = \{e_1, \dots, e_n\}$ be a set of H-path connected edges. Label the endpoints of each edge $e_i$ with $v_i, w_i$. We construct $G_k$ as follows: To each edge $e_i$, attach a complete graph on $k$ vertices, labeled $K^i$, so that every vertex of $K^i$ is adjacent to $v_i$ and $w_i$.
\end{construction}

We first prove some properties of the graphs $G_k$.

\begin{lemma}\label{lem: Gk not hamiltonian}
Let $G$ be a non-hamiltonian with a set of edges $S$ which is H-path connected.  For each $k \in \mathbb{N}$ the graph $G_k$ as in Construction \ref{construction of Gk} is non-hamiltonian.
\end{lemma}

\begin{proof} Assume (toward contradiction) that $G_k$ is hamiltonian. Then there exists a hamiltonian cycle $\lambda'$ in $G_k$. Up to cyclic permutations, we may assume that the first vertex of $\lambda'$ is $v_1$. By construction, we know that vertices in each $K^i$ are only adjacent to $v_i$ and $w_i$ in $G$. Therefore, between any pair of vertices $v_i$ and $w_i$, $\lambda'$ must pass through all the vertices of $K^i$. So up to reversing the orientation of $\lambda'$ we may assume that the next vertex of $\lambda'$ is in $K^1$, and the first vertex of $\lambda'$ after leaving $K^1$ is $w_1$.

Consider the path $\lambda$ in $G$ obtained by replacing every subpath of $\lambda'$ connecting $v_i$ to $w_i$ with the edge $e_i$. This is a hamiltonian path in $G$, and since the last vertex of $\lambda'$ was adjacent to $v_1$ and not in $K^1$, it is a hamiltonian cycle in $G$. But $G$ is not hamiltonian, so this is a contradiction.
\end{proof}

\begin{lemma}\label{lem: Gk has hamiltonian paths}
Let $G$ be non-hamiltonian, $S$ a set of H-path connected edges, and $G_k$ constructed as in Construction \ref{construction of Gk}. There is a hamiltonian path in $G_k$ from each vertex in $K^i$ to each vertex in $K^j$, for any $i \neq j$.
\end{lemma}

\begin{proof} Since $e_i, e_j \in S$ and $S$ is H-path connected, we know there exists a hamiltonian path in $G$ starting from an endpoint of $e_i$ and ending at an endpoint of $e_j$ which contains all the other edges in $S$. Call this hamiltonian path $\rho$. Then there exists a hamiltonian path $\rho'$ in $G_k$ obtained by replacing each edge $e_\ell$ with a hamiltonian path in $\{v_\ell, w_\ell\}\cup K^\ell$ (and possibly appending paths in $K^i\cup \{w_i\}$ and $K^j \cup \{v_j\}$ as needed).
\end{proof}

\begin{theorem}\label{thm: highly pair strung family}
If $G$ is a non-hamiltonian graph containing a set of $n>1$ edges which is H-path connected, then  $\displaystyle r_{G_k} \to \frac{n-1}{n}$.
\end{theorem}

\begin{proof}
Let $G$ be a non-hamiltonian graph of order $m$ containing an H-path connected set of $n$ edges. By Lemma \ref{lem: Gk not hamiltonian} the graphs $G_k$ are all non-hamiltonian, and by Lemma \ref{lem: Gk has hamiltonian paths} $G_k$ contains at least $\binom{n}{2}k^2 $ hamiltonian paths.

Thus we can bound $r_{G_k}$ from below as follows:
\begin{align*}
    \displaystyle r_{G_k} &\geq \frac{\displaystyle\binom{n}{2} k^2}{\displaystyle\binom{m+n k}{2}} \\
    \displaystyle &= \frac{n(n-1)k^2}{(m+nk)(m-1+nk)}\\
    \displaystyle &= \frac{n(n-1)k^2}{n^2 k^2 +2mnk - nk+m^2-m}.
\end{align*}
As $k$ approaches infinity, this approaches $\frac{n-1}{n}$.

On the other hand, we can also bound $r_{G_k}$ from above. Notice that there is no hamiltonian path connecting two vertices in the same complete subgraph $K^i$; indeed, if there were such a path then $G_k$ would contain a hamiltonian cycle. So
    \begin{align*}
     r_{G_k} &\leq \frac{\displaystyle\binom{m+nk}{2} -\binom{k}{2}n}{\displaystyle\binom{m+n k}{2}} \\
     &= \frac{(m+nk)(m-1+nk) - k(k-1)n}{(m+nk)(m-1+nk)}\\
     &= \frac{(n^2-n) k^2 +(2mn)k +m^2-m}{n^2 k^2 +2mnk - nk+m^2-m}.
    \end{align*}
As $k$ approaches infinity, this approaches $\frac{n-1}{n}$. So 
    \[
        \lim_{k\to\infty} r_{G_k} = \frac{n-1}{n},
    \]
as desired.
\end{proof}

Note that the expression $\frac{n(n-1)k^2}{n^2 k^2 +2mnk - nk+m^2-m}$ is increasing in $k$, so the asymptotic bound on $r_{G_k}$ may  not be achieved in a finite graph. However, if we are more careful in counting the number of pairs of vertices connected by hamiltonian paths, we may find a maximum at a finite value of $k$. We can see this explicitly by taking $G$ to be the Petersen graph.

\begin{proposition}
Let $P$ denote the Petersen graph, and let $\{P_k\}_{k \in \mathbb{N}}$ denote the family of graphs constructed as in Construction \ref{construction of Gk}. We have $r_{P_k} = \frac{10k^2 + 40k + 20}{\binom{10+5k}{2}}.$ In particular, $r_{P_k}$ is maximized at $k = 18$, where $r_{P_{18}}= .8\overline{04}.$
\end{proposition}

\begin{figure}[h]
\centering
	\begin{tikzpicture}

    \node at (7.75,0) {$v_1$};
    \node at (6.7,3.55) {$v_2$};
    \node at (10.5,5.5) {$v_3$};
    \node at (13.3,3.55) {$v_4$};
    \node at (12.25,0) {$v_5$};
    \node at (8.6,1.2) {$w_1$};
    \node at (11.5,3.5) {$w_4$};
    \node at (8.5,3.5) {$w_2$};
    \node at (11.4,1.2) {$w_5$};
    \node at (10.5,4) {$w_3$};
    
    \node[red] at (9, .5) {$e_1$};
    \node[red] at (7.75, 2.75) {$e_2$};
    \node[red] at (9.75, 4.5) {$e_3$};
    \node[red] at (12.5, 2.8) {$e_4$};
    \node[red] at (11, .5) {$e_5$};
    
    \draw[thick] (11.75,0) -- (8.25,0) -- (7,3.25) -- (10,5.5) -- (13,3.25) -- (11.75,0) -- (11,1) -- (10,4) -- (9,1) -- (11.5,3) -- (8.5,3) -- (11,1);
    \draw[thick] (8.25,0) -- (9,1);
    \draw[thick] (7,3.25) -- (8.5,3);
    \draw[thick] (10,5.5) -- (10,4);
    \draw[thick] (13,3.25) -- (11.5,3);
    \draw[red, very thick] (8.25,0) -- (9,1);
    \draw[red, very thick] (7,3.25) -- (8.5,3);
    \draw[red, very thick] (10,5.5) -- (10,4);
    \draw[red, very thick] (13, 3.25) -- (11.5, 3);
    \draw[red, very thick] (11.75,0) -- (11, 1);
    \draw[fill=black] (8.25,0) circle (3pt);
    \draw[fill=black] (7,3.25) circle (3pt);
    \draw[fill=black] (10,5.5) circle (3pt);
    \draw[fill=black] (11.5,3) circle (3pt);
    \draw[fill=black] (11.75,0) circle (3pt);
    \draw[fill=black] (9,1) circle (3pt);
    \draw[fill=black] (11,1) circle (3pt);
    \draw[fill=black] (8.5,3) circle (3pt);
    \draw[fill=black] (13,3.25) circle (3pt);
    \draw[fill=black] (10,4) circle (3pt);

\end{tikzpicture}
	\caption{The Petersen graph, with an H-path connected set of edges highlighted in red.}
	\label{fig: 45Petersen} 
\end{figure}

    \begin{proof}
        Consider the Petersen graph, $P$. Pick a perfect matching in $P$ with edges $\{e_1, e_2, \dots, e_5\}$. One can verify that this is an H-path connected set of edges. Label one vertex of $e_i$ by $v_i$ and the other vertex by $w_i$, as in Figure \ref{fig: 45Petersen}.  By Lemma \ref{lem: Gk not hamiltonian} $P_k$ is non-hamiltonian, and by Lemma \ref{lem: Gk has hamiltonian paths} there is a hamiltonian path connecting any pair of vertices in distinct subgraphs $K^i, K^j$.

    \begin{claim}\label{claim: peterson Ki to P}
        Let $x \in V(K^i)$ and $y\in V(P) - \{v_i, w_i\}$. There is a hamiltonian path in $P_k$ connecting $x$ to $y$.
    \end{claim}

    \begin{proof}
        We may assume that $x \in V(K^1)$. For each $i \in \{2, 3, 4, 5\}$ let $\gamma_i$ be a hamiltonian path in $K^i$. Let $\gamma_1$ be a path in $K^1$ containing every vertex except $x$. There are two cases: Either the shortest path from $y$ to a vertex in $K^1$ has edge length $2$, or it has edge length $3$.

        Suppose that the shortest path from $y$ to a vertex in $K^1$ has length $3$. We can assume without loss of generality that $y = w_2$. The following path is hamiltonian: 
            \[
            \gamma = x, \gamma_1, v_1, w_1, w_4, \gamma_4, v_4, v_5, \gamma_5, w_5, w_3, \gamma_3, v_3, v_2, \gamma_2, y.
            \]

        If instead the shortest path from $y$ to a vertex in $K^1$ has length $2$, we can assume without loss of generality that $y=v_2$. The following path is hamiltonian: 
            \[
                \gamma = x, w_1, \gamma_1, v_1, v_5, \gamma_5, w_5, w_3, \gamma_3, v_3, v_4, \gamma_4, w_4, w_2, \gamma_2, y.
            \]
    \end{proof}

    \begin{claim}
        Let $x$ a vertex in $V(K^i) \amalg \{v_i, w_i\}$ and $y$ a vertex in $V(K^i)$. There is no hamiltonian path between $x$ and $y$ in $P_k$.
    \end{claim}

    \begin{proof}
        Suppose $x, y \in V(K^i)\amalg \{v_i, w_i\}$. Since $x, y$ are adjacent, any hamiltonian path from $x$ to $y$ would imply the existence of a hamiltonian cycle. By Lemma \ref{lem: Gk not hamiltonian} this is impossible.
    \end{proof}

    \begin{claim}
        Let $x, y \in V(P)$, $k>0$. There is a hamiltonian path in $P_k$ between $x$ and $y$ if and only if 
            \begin{enumerate}
                \item $x, y$ are not adjacent in $P$, and
                \item exactly one of $x, y$ lies in $\{w_1, w_2, w_3, w_4, w_5\}$.
            \end{enumerate}
    \end{claim}

    \begin{proof}
        Suppose that $x, y$ are adjacent in $P$. By Lemma \ref{lem: Gk not hamiltonian} there is no hamiltonian path from $x$ to $y$.
        
        Suppose instead that $x$ and $y$ are not adjacent in $P$. We may assume that $x = v_1$ and $y \in V(P) - \{w_1, v_2, v_5\}.$  Up to symmetry of $P$ there are two cases: either $y = w_2$ or $y = w_3.$ Let $\gamma_i$ denote a hamiltonian path in $K^i$. 

        If $y = w_2$, the path 
            \[
            \gamma = x, \gamma_1, w_1, w_4, \gamma_4, v_4, v_5, \gamma_5, w_5, w_3, \gamma_3, v_3, v_2, \gamma_2, x
            \]
        is hamiltonian in $P_k$.

        If $y = w_3$, the path
            \[
            \gamma = x, \gamma_1, w_1, w_4, \gamma_4, v_4, v_5, \gamma_5, w_5, w_2, \gamma_2, v_2, v_3, \gamma_3, y
            \]
        is hamiltonian.

        Suppose that there exists a hamiltonian path $\lambda'$ in $P_k$ from $x$ to $y$. Note that for any $j$, the minimal subpath of $\lambda'$ containing all the vertices of $K^j$ must have edge length exactly $k-1$. Let $\gamma_j$ denote these subpaths in $\lambda'$. Note that $\lambda'$ thus includes either $v_i, \gamma_i, w_i$ or $w_i, \gamma_i, v_i$ as a subpath for each $i \in [5]$. Since $5$ is odd and the first vertex of $\lambda'$ is $v_1$, the final vertex of $\lambda'$ must be $w_j$ for some $j \in [5]$. Thus there is no hamiltonian path in $P_k$ from $v_i$ to $v_j$ for any $i, j \in [5]$. Similarly, there is no hamiltonian path in $P_k$ from $w_i$ to $w_j$ for any $i, j \in [5]$.

    \end{proof}

        This covers all the possible pairs of vertices, so for any $k>0,$ $P_k$ is $10k^2 + 40k + 20$-pair-strung, and it is not $k'$-pair-strung for any $k'>10k^2 + 40k + 20$. Therefore 
            \[
                r_{P_k} = \frac{10k^2 + 40k + 20}{\binom{10+5k}{2}}.
            \]
    \end{proof}

We now describe a family of graphs which contain large H-path connected sets.

\begin{theorem}\label{thm: bipartite h paths}
    Let $G= K_{n, n+1}$ be the complete bipartite graph with bipartitioned vertices $\{a_1, \dots, a_n\}\cup \{b_1, \dots, b_{n+1}\}$. Then the set 
        \[
            S = \{(a_1, b_1), \dots, (a_n, b_n)\}
        \]
    is an H-path connected set in $G$.
\end{theorem}
    \begin{proof}
        Let $1\leq i<j\leq n$. The path
            \begin{multline*}
                    b_i, a_i, b_{i-1}, a_{i-1}, \dots, b_1, a_1, b_{n+1}, a_n, b_n,\\
                    a_{n-1}, b_{n-1}, \dots, a_{j+1}, b_{j+1}, a_{i+1}, b_{i+1}, \dots, a_{j}, b_{j}.
                \end{multline*}
        is hamiltonian and contains all edges in $S$, as needed (see Figure \ref{fig: bipartite h path} for an example).
    \end{proof}

Putting together Theorem \ref{thm: highly pair strung family} and Theorem \ref{thm: bipartite h paths}, we get the following result, which provides an alternative proof that $R=1.$

\begin{cor}
    For all $n>1$, there is a family of graphs $G_k$ so that $r_{G_k}\to \frac{n-1}{n}$ as $k\to \infty$. In particular, $R = 1$.
\end{cor}

\begin{figure}
    \centering
    \includegraphics[width=.45\textwidth]{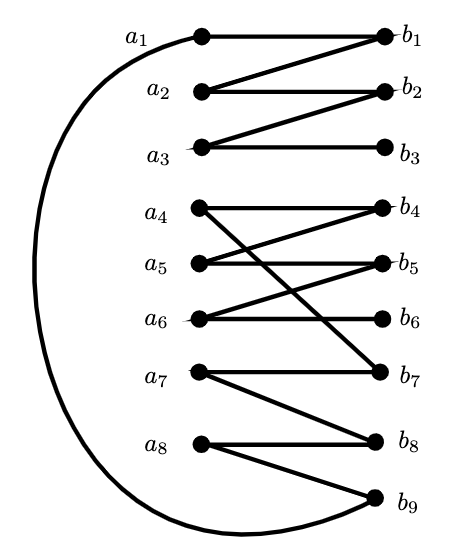}
    \caption{Hamiltonian path in $K_{8, 9}$ connecting edges $(a_3, b_3)$ and $(a_6, b_6)$. }
    \label{fig: bipartite h path}
\end{figure}

\section{A Minimal Hamiltonian-Connected Graph}\label{sec: minimal hc graphs}
In the previous section, we demonstrated the existence of graphs that are as close to being hamiltonian-connected as possible without being hamiltonian.  In this section, we explore an alternative extreme: we construct a family of hamiltonian-connected (and therefore hamiltonian) graphs that have a minimal number of edges and minimal number of minimum degree vertices. This partially answers a question of Modalleliyan and Omoomi (\cite{omoomi}), which asked whether there exist \emph{minimally hamiltonian-connected} graphs (i.e.\ hamiltonian-connected graphs such that the removal of any edge results in a non-hamiltonian-connected graph) with maximal vertex degree $\Delta$ so that $3\leq \Delta \leq \lceil n/2\rceil.$ While preparing this paper, Zhan \cite{Zhan} completely answered this question in the affirmative. In particular, he constructs a cubic graph which is minimally hamiltonian-connected. However, our construction is sufficiently distinct from Zhan's that we include it out of interest to the field. Our construction is highly symmetric and inductive.

We first note that a minimal hamiltonian-connected graph of order $n$ must have at least $\frac{3n}{2}$ edges. This was first proven by Moon \cite{moon}.

\begin{theorem}[\cite{moon}]
If $G$ is a hamiltonian-connected graph with $n>3$ vertices, then every vertex has degree at least $3$, and $G$ must have at least $\frac{3n}{2}$ edges.
\end{theorem}

\begin{proof}
Suppose that some vertex $v$ of $G$ has degree 2. Let $u, w$ be the two vertices adjacent to $v$. If $\gamma$ is a hamiltonian path in $G$ between $u$ and $w$, it contains $v$ and therefore contains $u, v, w$ as a subpath. Thus $\gamma = u,v,w$. But since $G$ has more than 3 vertices, $\gamma$ is not hamiltonian. Therefore every vertex of $G$ has degree at least 3, so the number of edges in $G$ is at least $\frac{3n}{2}.$
\end{proof}

Now we construct a family of graphs of even order $m=2n$ which have $3n$ edges. We will show that these graphs are hamiltonian-connected. Note that since every hamiltonian-connected graph is hamiltonian, such graphs must contain a hamiltonian cycle as a subgraph. So we begin our construction with the cycle on $2n$ vertices, $C_{2n}$, where $n\equiv 0\mod 3$.

Label the vertices of the cycle $C_{2n}$ by $0, 1, 2, \dots, 2n-1$ in order. Add edges $(0, n), (1, n+2), (2, n+1)$, and continue in this pattern in groups of three until every vertex has degree 3.

More precisely, for $n \equiv 0 \mod 3$, we define $\Gamma_n$ to be the graph with vertex set $V(\Gamma_n) = [2n]$ and edge set 
\[
E(\Gamma_n) = E(C_{2n})\cup \{ (3i, 3i+n), (3i+1, 3i+n+2), (3i+2, 3i+n+1)\}_{i=0}^{n-1}.\]
The graphs $\Gamma_3, \Gamma_6, \Gamma_9$ constructed in this way are illustrated in Figure \ref{fig: gamma369}.

\begin{figure}
    \centering
    \includegraphics[width=1\textwidth]{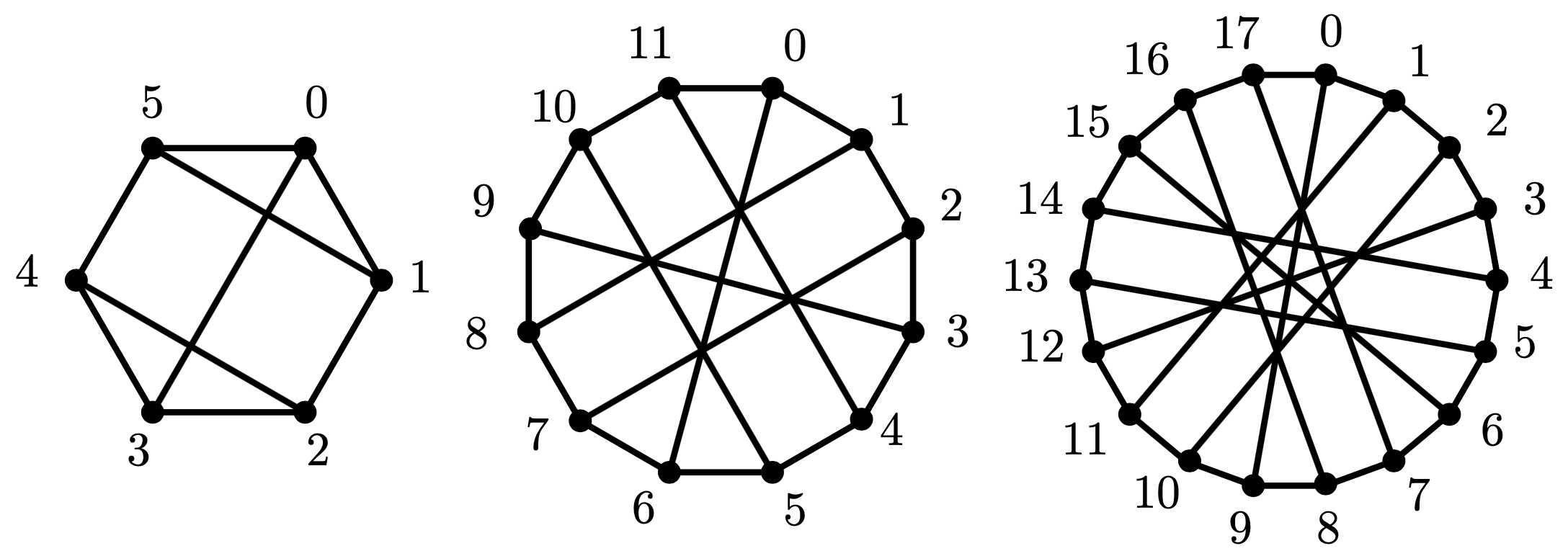}
    \caption{Shown here are $\Gamma_3, \Gamma_6$ and $\Gamma_9$.}
    \label{fig: gamma369}
\end{figure}

\begin{theorem}\label{thm: minimal hamiltonian-connected}
For any $n\equiv 0\mod 3$, $\Gamma_n$ is a hamiltonian-connected graph of order $2n$ with $3n$ edges.
\end{theorem}

One standard tool to demonstrate that a graph is hamiltonian-connected is to show that it has sufficiently many edges, or vertices of sufficiently high degree (c.f. \cite{Ore, Faudree, Wei, Kewen}). Since the graphs $\Gamma_n$ are sparse and contain no vertex of degree $>3$, none of these results apply. We will instead show that $\Gamma_n$ is hamiltonian-connected by considering several cases and inducting on $n$.

To emphasize the relationship between $\Gamma_n$ and $\Gamma_{n+3}$, we relabel the vertices of $\Gamma_{n+3}$ by $0,1, \dots, n-1, a, b,c, n, n+1, \dots, 2n-1, \alpha, \gamma, \beta $. By construction,
$\Gamma_{n+3}$ contains all the edges in 
$\Gamma_n$ except for $(n-1, n)$ and $(2n-1, 0)$, as well as edges $(a, \alpha), (b, \beta), (c, \gamma),$ $(n-1, a), (c, n), (2n-1, \alpha), (\beta,0).$ This is illustrated in Figure \ref{fig: inductive step}. Partition the vertex set of $\Gamma_{n+3}$ into \emph{new} vertices $N$, \emph{proximal} vertices $P$, and \emph{orbital} vertices $O$, where
    \begin{align*}
        N &= \{a, b, c, \alpha, \beta, \gamma\},\\
        P &= \{0, n-2, n-1, n, 2n-1, 2n-2\},\\
        O &= \{1, 2, \dots, n-3, n+1, n+2, \dots, 2n-3\}.
    \end{align*}

\begin{figure*}
    \centering
    \includegraphics[width=1\textwidth]{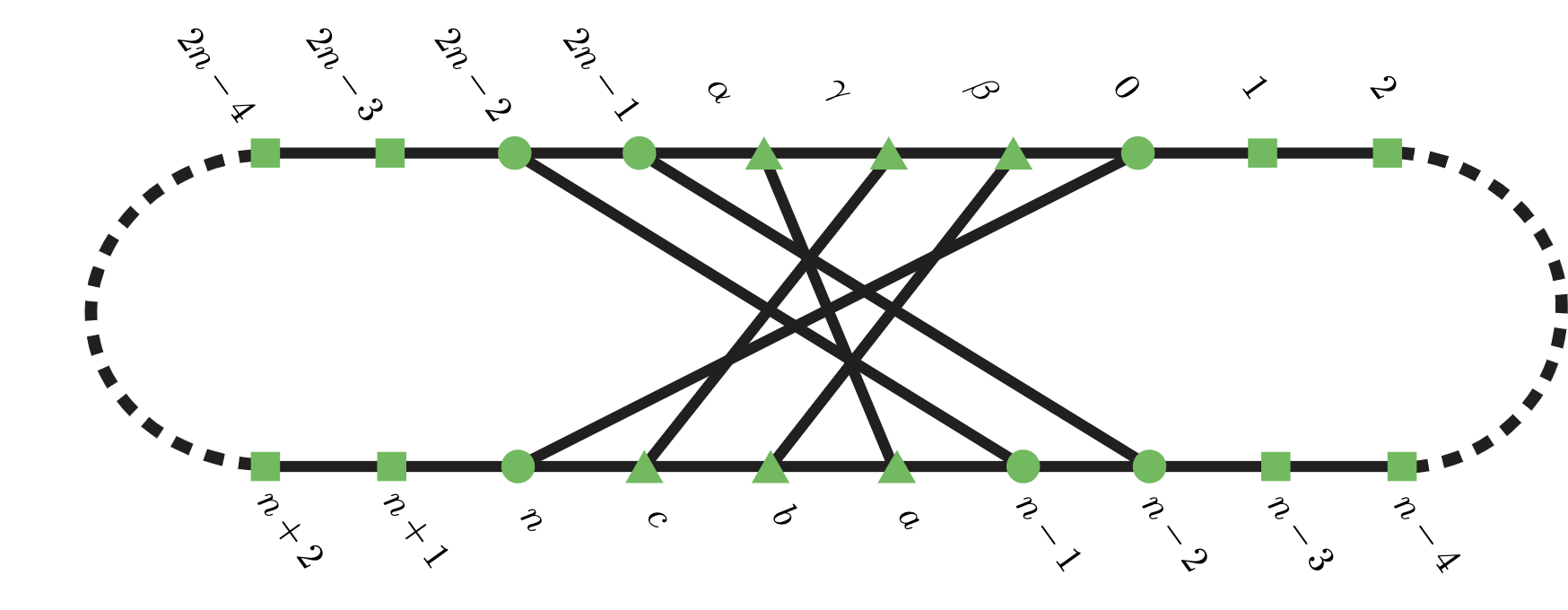}
    \caption{Part of the graph $\Gamma_{n+3}$ is shown here. It can be obtained from $\Gamma_n$ by inserting vertices $a, b, c, \alpha, \beta, \gamma$ and edges $(a, \alpha), (b, \beta), (c, \gamma)$, along with the edges needed to complete the cycle. Vertices indicated with squares are in $O$, with circles are in $P$, and with triangles are in $N$. }
    \label{fig: inductive step}
\end{figure*}

We first prove a special case. We will use rotational symmetry of $\Gamma_{n+3}$ and this special case to prove the other cases.

\begin{lemma}\label{case: old to old}
Suppose that $\Gamma_n$ is hamiltonian-connected, and $n\geq 12$. For any distinct vertices $u, v \in O$, there exists a hamiltonian path in $\Gamma_{n+3}$ from $u$ to $v$.
\end{lemma}

\begin{proof}
 By the inductive hypothesis, there exists a hamiltonian path $\gamma$ in $\Gamma_n$ from $u$ to $v$, which necessarily contains all of the vertices of $P$. 
 
 Since the six vertices of $P$ are not endpoints of $\gamma$, and since there are no edges from $\{2n-1, n-1\}$ to any vertex of $O$, $\gamma$ must enter and leave $P$ at vertices in $\{0, n-2, n, 2n-2\}$. Let $\gamma_P$ be the subpath(s) of $\gamma$ restricted to the induced subgraph of $\Gamma_n$ on $P$. Then $\gamma_P$ has endpoints in $\{0, n-2, n, 2n-2\}$, so it is a single connected path or two connected paths.
 
 We first suppose $\gamma_P$ is a single connected path. There are 6 possible pairs of endpoints of $\gamma_P$. In each case, there are up to 2 possible paths connecting the endpoints. For example, consider the endpoints $0$ and $2n-2$ (see Figure \ref{fig: ex1path}). Since $n, n-2$ are not endpoints of $\gamma_P$, the edges $(n, n+1)$ and $(n-2, n-3)$ are not contained in $\gamma$. Then since $\gamma_P$ contains $n$ and $n$ is not an endpoint of $\gamma$, $\gamma_P$ must include edges $(0, n)$ and $(n, n-1).$ Similarly, $\gamma_P$ must contain edges $(2n-1, n-2)$ and $(n-1, n-2)$. Since $\gamma_P$ contains $(n, n-1)$ and $(n-1, n-2)$, it can not contain $(n-1, 2n-2)$. Similarly, it does not contain $(2n-1, 0).$ Since $\gamma_P$ is a path, it must contain the edge $(2n-2, 2n-1)$, and in particular it must be the highlighted path indicated in Figure \ref{fig: ex1path}.

 One can make a similar argument for the other endpoint pairs. This leaves 6 possible paths for $\gamma_P$. By a similar argument, when $\gamma_P$ contains two connected subpaths there are 2 cases. These are all illustrated in Figure \ref{fig: 8paths}.
 
 \begin{figure*} 
     \centering
     \includegraphics[width=.5\textwidth]{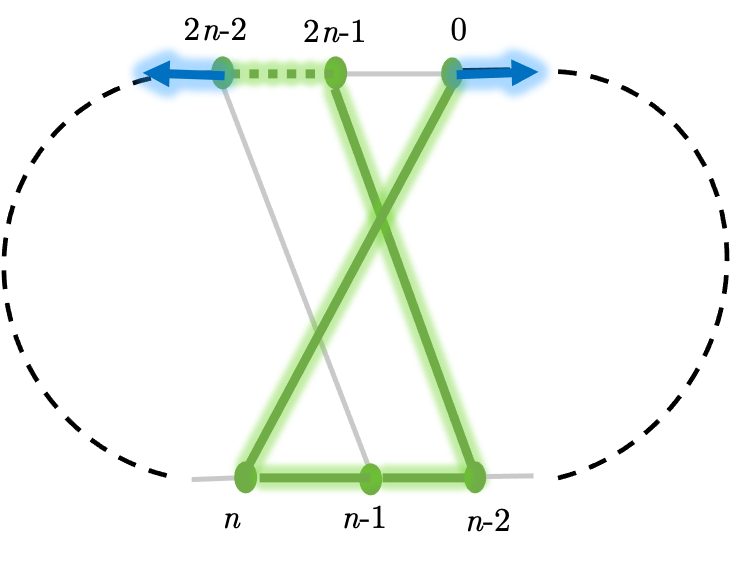}
     \caption{Example case of $0$ and $2n-2$ as endpoints of $\gamma_P$. The only possible path is shown, as described in Lemma \ref{case: old to old}}
     \label{fig: ex1path}
 \end{figure*}

 Each of the possibilities for $\gamma_P$ has a corresponding path $\gamma_P'$ in $\Gamma_{n+3}$, as shown in Figure \ref{fig: 8paths}, which meets every vertex in $N$. We can obtain a hamiltonian path in $\Gamma_{n+3}$ by substituting $\gamma_P$ in $\gamma$ with $\gamma_P'$.
 \end{proof}
 
 \begin{figure*}
     \centering
     \includegraphics[width=.6\textwidth]{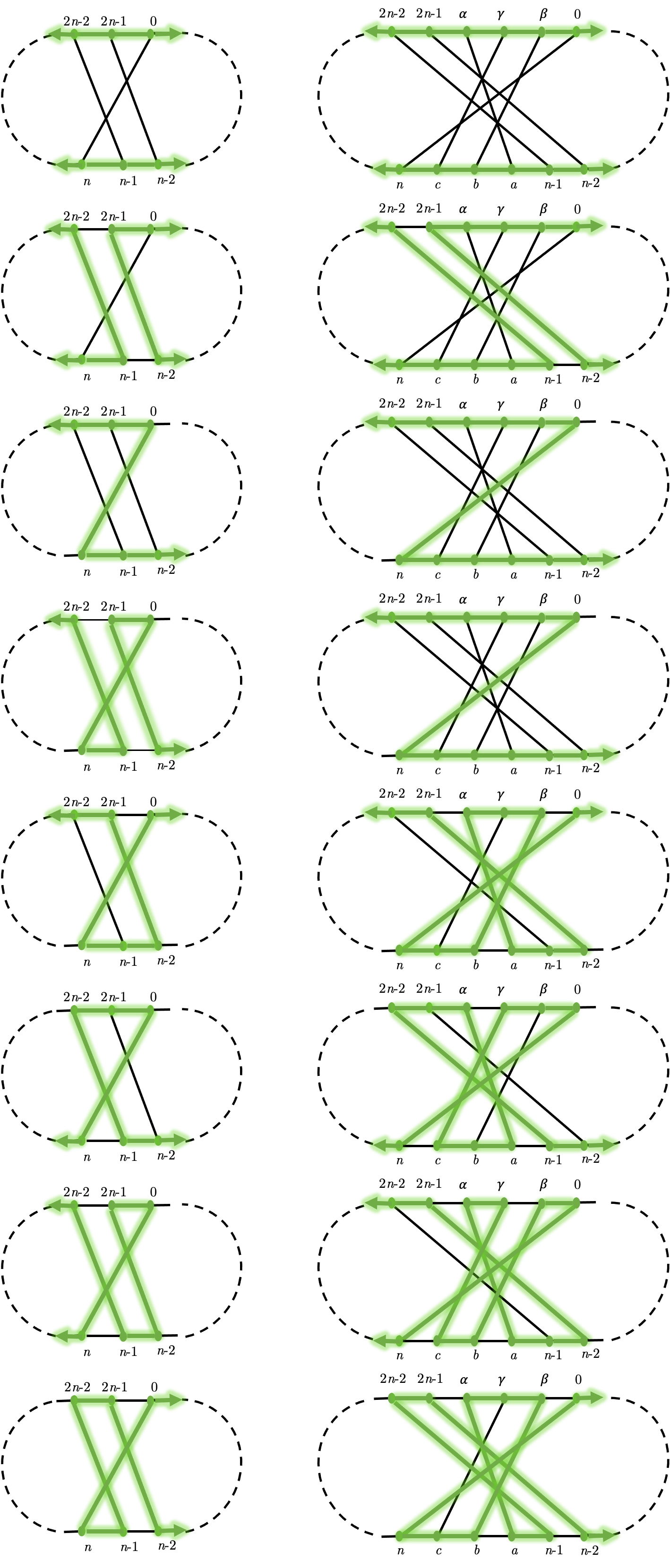}
     \caption{The 8 possibilities of $\gamma_P$ on the left and the corresponding paths $\gamma_P'$ in $\Gamma_{n+3}$ on the right, as described in Lemma \ref{case: old to old}.}
     \label{fig: 8paths}
 \end{figure*}

\begin{proof}[Proof of Theorem \ref{thm: minimal hamiltonian-connected}]
It can be verified (e.g.\ by computer program) that $\Gamma_k$ is hamiltonian-connected for $k = 3, 6, 9, 12.$ Suppose that $\Gamma_n$ is hamiltonian-connected. Let $u, v$ be a pair of distinct vertices in $\Gamma_{n+3}$. We will prove that there is a hamiltonian-path connecting them by considering several cases, depending on the vertex sets that $u, v$ belong to.

\begin{enumerate}
    \item[Case 1:]  $u, v \in O$. Apply Lemma \ref{case: old to old}.
    \item[Case 2:] $u, v \in N \cup P$. Rotate $\Gamma_{n+3}$ clockwise along the cycle $C_{2n+6}$ by 6 vertices. Note that this is a automorphism of $\Gamma_{n+3}$. Then both $u$ and $v$ are now in $O$, so there is a hamiltonian path $\gamma$ from $u$ to $v$ by Case 1.
    \item[Case 3:] $u\in O, v \in N$. If $v \in \{\alpha, a\}$, rotate $\Gamma_{n+3}$ counterclockwise by 3 vertices. Then $v \in O$, and either $u \in O$ or $u \in \{n, c, b, 0, \beta, \gamma\}$. If $u \in O$, then by Case 1 there is a hamiltonian path $v$ to $u$. Otherwise, rotate a further 6 vertices counterclockwise. Then $u \in O$ and $v \in O$, so by Case 1 we can find a $(u, v)$ hamiltonian path.
    \item[Case 4:] $u \in O, v \in P$. Rotate $\Gamma_{n+3}$ so that $v$ is in $N$. By the Cases 2 and 3 there is a $(u, v)$ hamiltonian path, regardless of the set containing $u$. 
\end{enumerate}
\end{proof}

While this theorem is restricted to graphs of order $2n$ where $n\equiv 0\mod 3$, a small adaptation to the construction produces graphs of arbitrary even order. In particular, to produce the graph of order $2(n-1)$, first produce the graph of order $2n$, then delete the edge $(\beta, b)$ and merge the two edges that share degree-2 vertices. To produce the graph of order $2(n-2)$ delete the edge $(\gamma, c)$ and merge the two edges that share degree-2 vertices. We believe that these are all hamiltonian-connected, and that a similar case-analysis approach will provide a proof. We have omitted this for the sake of space. We have verified hamiltonicity of these graphs for $n\leq 84$ by adapting the program \texttt{hamiltonicityChecker} (see \cite{jr_program, jr_paper}), which has methods for determining the existence of a hamiltonian path between two specified vertices.

\printbibliography

@article {HU21,
    AUTHOR = {Hu, Yanan and Zhan, Xingzhi},
     TITLE = {Regular homogeneously traceable nonhamiltonian graphs},
   JOURNAL = {Discrete Appl. Math.},
  FJOURNAL = {Discrete Applied Mathematics. The Journal of Combinatorial
              Algorithms, Informatics and Computational Sciences},
    VOLUME = {310},
      YEAR = {2022},
     PAGES = {60--64},
}

@article {omoomi,
    AUTHOR = {Modalleliyan, Maliheh and Omoomi, Behnaz},
     TITLE = {Critical {H}amiltonian connected graphs},
   JOURNAL = {Ars Combin.},
  FJOURNAL = {Ars Combinatoria. A Canadian Journal of Combinatorics},
    VOLUME = {126},
      YEAR = {2016},
     PAGES = {13--27},
}

@article {Zhan,
    AUTHOR = {Zhan, Xingzhi},
     TITLE = {The maximum degree of a minimally hamiltonian-connected graph},
   JOURNAL = {Discrete Math.},
  FJOURNAL = {Discrete Mathematics},
    VOLUME = {345},
      YEAR = {2022},
    NUMBER = {12},
     PAGES = {Paper No. 113159, 5},
}

@article {ClarkEntShapiro,
    AUTHOR = {Lane Clark and Roger Entringer and Henry Shapiro},
     TITLE = {Smallest maximally non-{H}amiltonian graphs. {II}},
   JOURNAL = {Graphs Combin.},
  FJOURNAL = {Graphs and Combinatorics},
    VOLUME = {8},
      YEAR = {1992},
    NUMBER = {3},
     PAGES = {225--231},
}

@article {ClarkEnt,
    AUTHOR = {Lane Clark and Roger Entringer},
     TITLE = {Smallest maximally non-{H}amiltonian graphs},
   JOURNAL = {Period. Math. Hung.},
  FJOURNAL = {Periodica Mathematica Hungarica. Journal of the J\'{a}nos Bolyai
              Mathematical Society},
    VOLUME = {14},
      YEAR = {1983},
    NUMBER = {1},
     PAGES = {57--68},
}

@article {Linetal,
    AUTHOR = {Lin, Xiaohui and Jiang, Wenzhou and Zhang, Chengxue and Yang, Yuansheng},
     TITLE = {On smallest maximally non-{H}amiltonian graphs},
   JOURNAL = {Ars Combin.},
    VOLUME = {45},
      YEAR = {1997},
     PAGES = {263--270},
}

@incollection {Chartrand,
    AUTHOR = {Chartrand, Gary and Gould, Ronald J. and Kapoor, S. F.},
     TITLE = {On homogeneously traceable non-{H}amiltonian graphs},
 BOOKTITLE = {Second {I}nternational {C}onference on {C}ombinatorial
              {M}athematics ({N}ew {Y}ork, 1978)},
    SERIES = {Ann. New York Acad. Sci.},
    VOLUME = {319},
     PAGES = {130--135},
      YEAR = {1979},
}

@article {Skupien,
    AUTHOR = {Skupie\'{n}, Zdzis\l aw},
     TITLE = {Homogeneously traceable and {H}amiltonian connected graphs},
   JOURNAL = {Demonstr. Math.},
  FJOURNAL = {Demonstratio Mathematica},
    VOLUME = {17},
      YEAR = {1984},
    NUMBER = {4},
     PAGES = {1051--1067},
}

@article {Ore,
    AUTHOR = {Ore, \O{}ystein},
     TITLE = {Hamilton connected graphs},
   JOURNAL = {J. Math. Pures Appl. (9)},
  FJOURNAL = {Journal de Math\'{e}matiques Pures et Appliqu\'{e}es. Neuvi\`eme S\'{e}rie},
    VOLUME = {42},
      YEAR = {1963},
     PAGES = {21--27},
}

@article {Zamfi,
    AUTHOR = {Zamfirescu, Carol T.},
     TITLE = {On hypohamiltonian and almost hypohamiltonian graphs},
   JOURNAL = {J. Graph Theory},
  FJOURNAL = {Journal of Graph Theory},
    VOLUME = {79},
      YEAR = {2015},
    NUMBER = {1},
     PAGES = {63--81},
}

@article {Faudree,
    AUTHOR = {Faudree, Ralph and Gould, Ronald and Jacobson, Michael
              and Schelp, Richard},
     TITLE = {Neighborhood unions and {H}amiltonian properties in graphs},
   JOURNAL = {J. Comb. Theory. Ser. B},
  FJOURNAL = {Journal of Combinatorial Theory. Series B},
    VOLUME = {47},
      YEAR = {1989},
    NUMBER = {1},
     PAGES = {1--9},
}

@incollection {Wei,
    AUTHOR = {Wei, Bing},
     TITLE = {Hamiltonian paths and {H}amiltonian connectivity in graphs},
      NOTE = {Graph theory (Niedzica Castle, 1990)},
   JOURNAL = {Discrete Math.},
  FJOURNAL = {Discrete Mathematics},
    VOLUME = {121},
      YEAR = {1993},
    NUMBER = {1-3},
     PAGES = {223--228},
}

@article {Kewen,
    AUTHOR = {Kewen, Zhao and Lai, Hong-Jian and Zhou, Ju},
     TITLE = {Hamiltonian-connected graphs with large neighborhoods and
              degrees},
   JOURNAL = {Missouri J. Math. Sci.},
  FJOURNAL = {Missouri Journal of Mathematical Sciences},
    VOLUME = {24},
      YEAR = {2012},
    NUMBER = {1},
     PAGES = {54--66},
}

@article{moon, 
    title={On a Problem of Ore}, 
    volume={49}, DOI={10.2307/3614234}, 
    number={367}, 
    journal={Math. Gaz.}, 
    publisher={Cambridge University Press}, 
    author={Moon, John W.}, 
    year={1965}, 
    pages={40–-41}
}

@software{jr_program,
  author = {Jan Goedgebeur and Jarne Renders and Gábor Wiener and Carol T. Zamfirescu},
  title = {K2-{H}amiltonian Graphs},
  url = {https://github.com/JarneRenders/K2-Hamiltonian-Graphs},
  version = {1},
  date = {2022-05-27},
}

@article{jr_paper,
    author = {Jan Goedgebeur and Jarne Renders and Gábor Wiener and Carol T. Zamfirescu},
    title = {K2-{H}amiltonian Graphs: II},
    year = {2024},
    journal = {J. Graph Theory},
    volume = {105},
    number = {4},
    pages = {580-611},
}

\end{document}